\documentclass{au}

\usepackage{amsmath,amssymb,url}
\usepackage[T1]{fontenc} 
\usepackage{tikz}
\usepackage{graphicx}
\usetikzlibrary{calc} 
\theoremstyle{plain}
\newtheorem{theorem}{Theorem}[section]
\newtheorem{lemma}[theorem]{Lemma}
\newtheorem{corollary}[theorem]{Corollary}
\newtheorem{proposition}[theorem]{Proposition}

\begin{document}
   
\title{Generating all finite modular 
lattices of a given size} 

\author[P. Jipsen]{Peter Jipsen}
\email{jipsen@chapman.edu}
\urladdr{http://www1.chapman.edu/~jipsen}
\address{Chapman University, Orange, CA, United States}

\author[N. Lawless]{Nathan Lawless}
\email{lawle108@mail.chapman.edu}
\urladdr{http://www.nlawless.com}
\address{Chapman University, Orange, CA, United States}

\dedicatory{Dedicated to Brian Davey on the occasion of his 65th birthday}



\begin{abstract}
Modular lattices, introduced by R.~Dedekind, are an important
subvariety of lattices that includes all distributive lattices.
Heitzig and Reinhold \cite{HR02} developed an algorithm to enumerate,
up to isomorphism, all finite lattices up to size 18. Here we adapt
and improve this algorithm to construct and count modular lattices up
to size 24, semimodular lattices up to size 22, and lattices of
size 19. We also show that $2^{n-3}$ is a lower bound for the number
of nonisomorphic modular lattices of size $n$.
\end{abstract}

\maketitle

\section{Introduction}

Enumeration of finite mathematical structures is an important tool
since it allows testing new hypotheses and searching for
counterexamples.  Additionally, it provides insight into the
properties of these structures. Here we concentrate on constructing,
up to isomorphism, all modular lattices with a given number of
elements. The algorithm we develop is a modification of the approach
of Heitzig and Reinhold~\cite{HR02} who enumerated (up to isomorphism)
all lattices with up to 18 elements.  The number of distributive
lattices of size up to 49 were calculated by Ern\'e, Heitzig and
Reinhold~\cite{EHR02}. In the Online Encyclopedia of Integer Sequences
(\texttt{oeis.org}) the relevant sequences are A006981, A006966 and
A006982, but the sequence for the number of modular lattices was given
only up to $n=11$. For $n=12$ there are 766 nonisomorphic modular
lattices, as was reported in~\cite{BV10}. We extend this result to
$n=24$ and also count the number of semimodular lattices up to size
$n=22$ (see Table~\ref{table1}).

Our algorithm uses an improved method for removing isomorphic copies,
which allowed us to recalculate the numbers in~\cite{HR02} for all 
lattices up to $n=18$ and go one step further to find the number of
nonisomorphic lattices with 19 elements. The calculations were
done on a cluster of 64 processors and took 26 hours for $n=18$ and
19 days for $n=19$.

In the remainder of this section, we define some properties and recall
some basic results of (semi)modular lattices. 
In Section 2, we give an outline of the algorithm used by \cite{HR02} to 
generate finite lattices up to isomorphism.
Then, in Section 3, we adapt this algorithm to generate modular lattices up 
to isomorphism by adding a series of constraints to the algorithm.
Section 4 contains an improvement for the algorithm used by \cite{HR02} by 
employing the canonical construction path introduced in \cite{McK98}.
In Section 5, the algorithm is adjusted to generate only vertically 
indecomposable modular lattices.

A \emph{modular lattice} $L$ is a lattice which satisfies the modular law
\begin{align*} a\geq c \text{ implies }a\wedge(b\vee c)=(a\wedge b)\vee 
c\text{ for all }a,b,c\in L.\end{align*}

Weaker conditions of modularity are semimodularity and lower semimodularity. As
usual, we write $a\prec b$ if $a$ is \emph{covered by} $b$.
        
A lattice $L$ is \emph{semimodular} if for all $a,b\in L$
\begin{align*}
a\wedge b\prec a,b \text{ implies that } a,b\prec a\vee b.
\end{align*}

Dually, $L$ is \emph{lower semimodular} if for all $a,b\in L$
\begin{align*}
a,b\prec a\vee b \text{ implies that }a\wedge b\prec a,b.
\end{align*}
Recall that a \emph{chain} in a lattice $L$ is a subset of $L$ such
that all elements in the subset are comparable.  We say that a lattice
has \emph{finite length} if all chains in it have finite cardinality.  
The next two well-known results below can be 
found for example in \cite{Gr98}.
	
	\begin{proposition}\label{mod-semimod}
A lattice of finite length is modular if and only if it is semimodular 
and lower semimodular.
	\end{proposition}
	
A chain $C$ in a poset $P$ is \emph{maximal} if whenever $C\subseteq
D\subseteq P$ and $D$ is a chain in $P$, then $C=D$. In a finite
lattice, a maximal chain is a chain from bottom to top such that each
element in the chain, other than the top, is covered by some element
in the chain.

\begin{theorem}[Jordan-H\"{o}lder Chain Condition]\label{JHCC}
Let $L$ be a finite semimodular lattice. Then, for any maximal
chains $C$ and $D$ in L, $|C|=|D|$.
\end{theorem}

\section{Generating finite lattices}
There are many ways to represent finite lattices and to construct bigger
lattices from smaller lattices. An algorithm that constructs up to isomorphism 
all combinatorial objects of a certain kind and of a given size is called
an \emph{orderly algorithm} if it produces exactly one member of each 
isomorphism class without testing 
that this member is nonisomorphic to previously constructed objects.
Such algorithms were first introduced by Faradzhev~\cite{Fa76} and 
Read~\cite{Re78} for enumerating finite graphs. 
Heitzig and Reinhold \cite{HR02} developed an orderly
algorithm to enumerate all finite lattices up to isomorphism and used
it to count the number of lattices up to size 18. Since our first algorithm
for modular lattices is based on their approach, we recall some of the details
here.

Let $L$ be a lattice. A nonempty antichain $A\subseteq L \setminus
\{0\}$ is a \emph{lattice-antichain} if $a\wedge b\in\{0\}\cup
     {\uparrow} A$ for all $a,b\in {\uparrow} A$, where
     ${\uparrow}A=\{b\mid b\geq a$ for some $a\in A\}$.  A finite lattice is
     called an $n$-lattice if its set of elements is
     $\{0,1,2,\ldots,n-1\}$, where $0$ and $1$ are the bottom and top
     elements.
	
Given a lattice antichain $A$ and an $n$-lattice $L$, a poset $L^A$
with $n+1$ elements is constructed by adding an element $n$ to $L$ as
an atom with $A$ as the set of its covers.  Furthermore, the following
lemma states that $L^A$ is a lattice.

\begin{lemma}[\cite{HR02}]
A subset $A\subseteq L\setminus \{0\}$ of an $n$-lattice $L$ is a
lattice-antichain if and only if $L$ is a subposet of an
$(n+1)$-lattice $L^A$ in which the element $n$ is an atom and $A$ is
the set of its covers.
\end{lemma}

In order to generate only one copy of each lattice up to isomorphism, the
weight $w(L)=(w_2(L),\ldots,w_{n-1}(L))$ of an $n$-lattice $L$ is
defined by setting $w_i(L)=\displaystyle\sum_{i\prec j} 2^{j}$.

With this weight, for two $n$-lattices $L$ and $M$, $w(L)$ is said to be
(lexicographically) smaller than $w(M)$ if there is an $i\leq n-1$
such that $w_i(L)<w_i(M)$ and $w_k(L)=w_k(M)$ for all
$k=2,\ldots,i-1$.  An $n$-lattice C is called a canonical lattice if
there is no $n$-lattice isomorphic to $C$ that has a smaller weight.
In order to check whether an $n$-lattice $L$ is canonical, one has to
check whether there is a permutation of the elements of $L$ that
yields an isomorphic copy of $L$ with a smaller weight.

With these definitions, a recursive algorithm is formulated in
\cite{HR02} which generates exactly all canonical lattices of order
less or equal to $n$ for a given natural number $n\geq 2$.

\begin{quote}
\texttt{
\!\!\!next\_lattice(integer $m$, canonical $m$-lattice $L$)\\
begin\\
\hspace*{.3in}if $m<n$ then\\
\hspace*{.6in}for each lattice-antichain $A$ of $L$ do\\
\hspace*{.9in}if $L^{A}$ is a canonical lattice then\\
\hspace*{1.2in}next\_lattice($m+1$, $L^{A}$)\\
\hspace*{.3in}if $m=n$ then output $L$\\
end}
\end{quote}
\begin{center}
Algorithm 1
\end{center}

The set of all maximal elements in a finite poset $P$ is called the
\emph{first level} of $P$ and is denoted by $lev_1(P)$. The
\emph{$(m+1)$-th level} of $P$ is recursively defined by
\begin{align*}lev_{m+1}(P)=lev_1(P\setminus \bigcup_{i=1}^m lev_i(P)).\end{align*}

Following \cite{HR02} we define $dep_P(p)$ to be the number $k$ such
that $p\in lev_k(P)$. Although Heitzig and Reinhold refer to this as
the depth of $p$, it is more traditional to consider the depth of $p$
to be given by $dep_P(p)-1$. To avoid confusion, we only use the
function $dep_P$ rather than the notion of depth.

We say an $n$-lattice $L$ is \emph{levelized} if 
\begin{align*}dep_L(i)\leq dep_L(j) \text{ for all } i,j\in L\setminus\{0\} \text{ with } i\leq j. \end{align*}
In other words, the levels form a partition on $L \setminus
\{0\}=\{1,2,\ldots,n-1\}$ of the form $\{1\mid2,\ldots,m_2\mid
m_2+1,\ldots,m_3\mid\ldots\mid m_{k-1}+1,\ldots,m_k\}$, where $k$ is
the number of levels in $L$.

Throughout the rest of the paper, we consider the bottom level to be
$lev_k(L)$, unless indicated otherwise. The following lemma gives us
an important property when generating levelized lattices, since it
tells us we only need to consider lattice-antichains that have at least
one element in the two bottom levels.

\begin{lemma}[\cite{HR02}]\label{L2.2}
For a levelized $n$-lattice $L$ and a lattice-antichain $A$, $L^A$ is levelized
if and only if $A\cap(lev_{k-1}(L)\cup lev_k(L))\neq \emptyset$.
\end{lemma}

   \section{Generating finite modular lattices}
In order to construct only modular lattices of size $n$ using this
algorithm, we start by selecting only the lattices that are modular
when we get to size $n$. However, modular lattices constitute a very
small fraction of the total number of lattices.  Therefore, it is
important to add constraints in order to minimize the generation of
non-modular lattices.  In order to do this, we present a series of
results to decide when a subtree in a search tree can be cut off, and
which lattice antichains must be considered in each step. During this
section, we refer to \emph{descendants} of a lattice $L$ as those
lattices generated through the element extension described in
\cite{HR02}, together with any additional constraints introduced in
this section.

\begin{lemma}\label{L3.1}
For an $n$-lattice $L$, and a lattice-antichain $A\subseteq L$,
if there exist $a,b\in A$ such that $dep_L(a)\neq dep_L(b)$, then all
descendants of $L^A$ are non-semimodular. Specifically, they are
non-modular.
\end{lemma}
\begin{proof}
Assume without loss of generality that $dep_L(a)<dep_L(b)$. Let $C_a$
and $C_b$ be the chains of maximal cardinality from 1 to $a$ and $b$
respectively.
	
For any $x\in L$, $dep_L(x)$ is equal to the cardinality of the
longest chain from $x$ to 1, hence $|C_a|<|C_b|$.
	
Next, in $L^A$, we have $n\prec a$ and $n\prec b$. Let $M$ be a
descendant of $L^A$, and choose any chain $D$ from 0 to $n$.
	
Then $D_a:=D\cup C_a$ and $D_b:=D\cup C_b$ are maximal chains of
different cardinality since $|D_a|=|D|+|C_a|<|D|+|C_b|=|D_b|$.  By
Theorem~\ref{JHCC}, it follows that $M$ is not semimodular, and
therefore is non-modular.
\end{proof}

From Lemmas \ref{L2.2} and \ref{L3.1} we may conclude the following result.

\begin{corollary}\label{levk}
For the construction of (semi-)modular lattices using the
orderly algorithm of \cite{HR02}, it is sufficient to consider
lattice-antichains $A$ such that $A\subseteq lev_{k-1}(L)$ or
$A\subseteq lev_k(L)$.
\end{corollary}

\begin{lemma}
Let $L$ be an $n$-lattice where $k=dep_L(n-1)$ is the bottom
nonzero level, and let $A\subseteq lev_k(L)$ be a lattice-antichain of
$L$.  If there is an atom of $L$ in $lev_{k-1}(L)$ then all descendants of
$L^A$ are non-semimodular, and hence non-modular.
\end{lemma}
\begin{proof}
Let $b\in lev_{k-1}(L)$ be an atom of $L$, and choose any
$a\in A\subseteq lev_{k}(L)$.  Then there exist chains
$C_a$ from $a$ to 1 and $C_b$ from $b$ to 1 of cardinality $k$ and $k-1$
respectively.
	
Since the new element $n$ is in a new level
$lev_{k+1}(L^A)$, and $dep_{L^A}(b)=k-1=(k+1)-2$, $b$ is contained in
the third lowest level of $L^A$. Therefore, by Corollary 3.2, it is not used
in the generation of any descendants, and $b$ remains as an atom in
all descendants.  Hence, the maximal chain $D_b:=\{0\}\cup C_b$ is of
constant cardinality $1+(k-1)=k$ for any descendant of $L^A$.
	
Let $M$ be a descendant of $L^A$. Choose any chain $C_n$ from 0 to
$n$, then $|C_n|\geq 2$. Therefore, for the maximal chain
$D_a:=C_n\cup C_a$,
	$$
		|D_a|=|C_n|+|C_a|\geq 2+k>k=|D_b|.
	$$ 
By Theorem \ref{JHCC}, since both $D_a$ and $D_b$ are maximal chains,
it follows that $M$ is non-semimodular and therefore non-modular.	
\end{proof}

An observation that significantly decreases the search space is based
on the following property of the algorithm: since elements are always
added below a lattice antichain, if two elements in the antichain fail
semimodularity, then those two elements also fail
semimodularity in any of the descendants.  Therefore, when adding a
new element below a lattice antichain, we should check that we are not
generating a non-semimodular lattice.

\begin{lemma}
For an $n$-lattice $L$ and a lattice antichain $A\subseteq L$, if
there exist $a,b\in A$ which do not have a common cover, then all
descendants of $L^A$ are non-semimodular.
\end{lemma}
\begin{proof}
In $L^A$, for the new element $n$, $n\prec a$ and $n\prec b$. However,
$a\nprec a\vee b$ or $b\nprec a\vee b$. Therefore, $L^A$ is not
semimodular.  Furthermore, for any descendant $M$ of $L^A$, it is not
possible to add a common cover to $a,b$. Hence, $M$ is not semimodular
(and consequently, not modular).
\end{proof}

Similarly, we can consider when it is not possible to make a non-lower
semimodular lattice into a lower semimodular lattice.

\begin{lemma}
Let $L$ be an $n$-lattice, and let $k$ be its bottom non-zero
level. If there exist $a,b\in lev_{k-2}(L)$ which do not satisfy lower
semimodularity, then all descendants of $L$ are non-lower semimodular
(and hence non-modular).
\end{lemma}
\begin{proof}
Given an $a,b$ such that $a,b\prec a\vee b$ but $a\wedge b\nprec a$ or
$a\wedge b \nprec b$, the algorithm can make $a,b$ satisfy lower
semimodularity by adding an element below $a,b$.  However, by
Corollary \ref{levk}, we only consider lattice antichains in
$lev_k(L)$ and $lev_{k-1}(L)$. Therefore, if $a,b\in lev_{k-2}(L)$, we
cannot add a common co-cover, and all descendants $M$ of $L$ are
non-lower semimodular.
\end{proof}

This lemma can be incorporated into the algorithm by checking that all
elements of $lev_{k-1}(L)$ satisfy lower semimodularity each time a new level
is added.

The preceding results are summarized in the following theorems.

\begin{theorem}
When generating semimodular lattices, for a lattice $L$, we only
consider lattice-antichains $A$ which satisfy all of the following
conditions:
\begin{itemize}
\item[(A1)] $A\subseteq lev_{k-1}(L)$ or $A\subseteq lev_k(L)$. 
\item[(A2)] If $A\subseteq lev_k(L)$, there are no atoms in $lev_{k-1}(L)$.
\item[(A3)] For all $x,y\in A$, $x$ and $y$ have a common cover.
\end{itemize}
\end{theorem}

\begin{theorem}
When generating modular lattices, for a lattice $L$, we only consider
lattice-antichains $A$ which satisfy \textup{(A1), (A2), (A3)} and
\begin{itemize}
\item[(A4)] If $A\subseteq lev_k(L)$, then $lev_{k-1}(L)$ satisfies lower
  semimodularity (i.~e., for all $x,y\in lev_{k-1}(L), \ x,y\prec
  x\vee y$ implies $x\wedge y\prec x,y$).
\end{itemize}
\end{theorem}
	 
Another improvement can be implemented in the last step when
generating lattices of size $n$ from those of size $n-1$, by only
considering lattice-antichains $A\subseteq lev_{k-1}(L)$ and
$A=lev_k(L)$. This is due to the following result.

\begin{lemma}\label{bottomlev}
For an $n$-lattice $L$ and a lattice antichain $A\subsetneq lev_k(L)$,
the $n+1$-lattice $L^A$ is non-modular.
\end{lemma}
\begin{proof}
Since $A\subsetneq lev_k(L)$, there exists $b\in lev_k(L)$ such that
$b\not \in A$. Let $a\in A$.  Since $a,b\in lev_k(L)$, there exist
chains $C_a$ and $C_b$ from $a$ to 1 and $b$ to 1 respectively, both of
cardinality $k$.
	
In $L^A$, $n\prec a$, but $n\nprec b$. Thus, for the maximal chains
$D_a:=\{0,n\}\cup C_a$ and $D_b:=\{0\}\cup C_b$,
\begin{align*} |D_a|=2+k>1+k=|D_b|,\end{align*}
thus $L^A$ is non-modular.
\end{proof}

\section{Dealing with isomorphisms}
When generating finite (modular) lattices using Algorithm~1,
the majority of the time is spent in testing if the lattice $L^A$ is
canonical, an operation of order $O(n!)$. An approach that
speeds-up the algorithm significantly, while still generating
exactly one isomorphic copy of each (modular) lattice is via
generation by \emph{canonical construction path}, which was introduced by
McKay \cite{McK98}.

This canonical construction has two components. The first is to
use only one representative of each orbit in the lattice antichains of
$L$. In other words, if there is an automorphism $g$ on $L$ such that
$\{g(a)\mid a\in A\}=B$ for lattice-antichains $A,B$, only one of these
antichains is chosen arbitrarily.

The second is, after the extension of any lattice $L$ using $A$, $L^A$ is
checked to see if $L$ is the inverse through a ``canonical deletion''.
This uses the canonical labeling of the program \texttt{nauty}
\cite{MP13}.  
In general, a \emph{canonical labeling} associates with each $n$-lattice
$L$ a permutation $c_L$ on $\{0,\ldots,n-1\}$ such that for any $n$-lattice
$M$ we have $L\cong M$ if and only if 
\[
\{(c_L(x),c_L(y))\mid x\le y\text{ in }L\}=
\{(c_M(x),c_M(y))\mid x\le y\text{ in }M\},
\]
i.e., the permutation maps each lattice to a fixed representative of its 
isomorphism class.
When a new $(n+1)$-lattice $L^A$ is generated from $L$ and
$A$, a canonical labeling $c_{L^A}$ of $L^A$ is generated using a partition
by levels in \texttt{nauty}. Let $n':=c_{L^A}^{-1}(n)$ denote the element
which maps to $n$ under the canonical labeling. We consider the set
$m(L^A)=\{\langle L^A, a\rangle \mid f(a)=n', f\in Aut(L^A)\}$, where
$\langle L^A, a\rangle$ denotes the lattice obtained by removing $a$
from $L^A$. Note that $L=\langle L^A,n\rangle$. If $L\in m(L^A)$, we 
say $L^A$ is \emph{canonical} and keep it, otherwise it is discarded.

Using this construction, Theorem~1 in \cite{McK98} states that
starting from any lattice, exactly one isomorphic copy of each descendant 
will be output. Thus, starting with the two-element lattice, we can generate
exactly one isomorphic copy of each lattice of a given size $n$. This
has an advantage over the construction used in \cite{HR02} since it
does not require checking all permutations of a lattice, and it uses
canonical labeling by \texttt{nauty}, which is generally considered
the most efficient canonical labeling program for small combinatorial
structures. Furthermore, this construction is orderly since it only
considers the lattices $L$ and $L^A$. This is beneficial during
computations because it does not require storage of previously generated
lattices or communication between nodes during parallel computations.
Given this, Algorithm~1 can be modified:

\begin{quote}
\texttt{
\!\!\!next\_lattice2(integer $m$, canonical $m$-lattice $L$)\\
begin\\
\hspace*{.3in}if $m<n$ then\\
\hspace*{.6in}$LAC:=\{A \mid A$ is a lattice-antichain of $L$\}\\
\hspace*{.6in}for each orbit $O$ of the action of $Aut(L)$ on $LAC$\\
\hspace*{.9in}select any $A\in O$\\
\hspace*{.9in}$c:=$ canonical labeling of $L^A$\\
\hspace*{.9in}$n':=c^{-1}(n)$\\
\hspace*{.9in}if $f(n)=n'$ for some $f\in Aut(L^A)$ then\\
\hspace*{1.2in}next\_lattice2($m+1$, $L^{A}$)\\
\hspace*{.3in}if $m=n$ then output $L$\\
end}
\end{quote}
\begin{center}
Algorithm 2
\end{center}
 
 \section{Vertically indecomposable modular lattices}
We say a lattice $L$ is \emph{vertically decomposable} if it contains
an element which is neither the greatest nor the least element of $L$
but is comparable with every element of $L$.  A lattice which is not
vertically decomposable is said to be \emph{vertically
indecomposable}.

Let $m^v(n)$ be the number of unlabeled vertically indecomposable
modular lattices. Then the recursive formula \cite{HR02} can be used
to compute the number of unlabeled modular lattices from the number of
unlabeled vertically indecomposable modular lattices.
$$m(n)=\displaystyle\sum_{k=2}^n m^v(k)\cdot m(n-k+1), \quad  n\geq 2$$

In order to avoid generating vertically decomposable modular lattices,
we only need to avoid using $lev_k(L)$ as a lattice antichain of $L$,
since then $n\in L^A$ would be comparable to all elements in $L^A$.
However, not using these lattice antichains would cut off branches of
the canonical path that could potentially generate
vertically indecomposable canonical lattices. 
Lemma~\ref{vi} tells us we can safely avoid using them
when $|lev_k(L)|=1$.
	
\begin{lemma}\label{vi} 
Given an $n$-lattice $L$ with only one atom $n-1$, then all
descendants of $L^{\{n-1\}}$ are vertically decomposable.
\end{lemma}
\begin{proof}
It is clear that $lev_k(L^A)=\{n\}$ and $lev_{k-1}(L^A)=\{n-1\}$,
where $n\leq n-1$. Under our construction, only lattice-antichains
$\{n\}$ and $\{n-1\}$ are considered.  Therefore, for any descendant
$M$ of $L^A$ and any new element $m\in M$ such that $m\not\in L^A$ or $m=0$,
$m\leq n-1$. Additionally, for all $a\in L\setminus\{0\}$, $n-1\leq a$. Thus, $M$ is
vertically decomposable.
\end{proof}

This means that we only construct vertically decomposable
lattices where the only comparable element is a single atom. However,
these are ignored during the count of vertically indecomposable
lattices.
	
Note that in the last step, by
Lemma~\ref{bottomlev}, we only have to consider
lattice-antichains in $lev_{k-1}(L)$.

\hspace{-30mm}
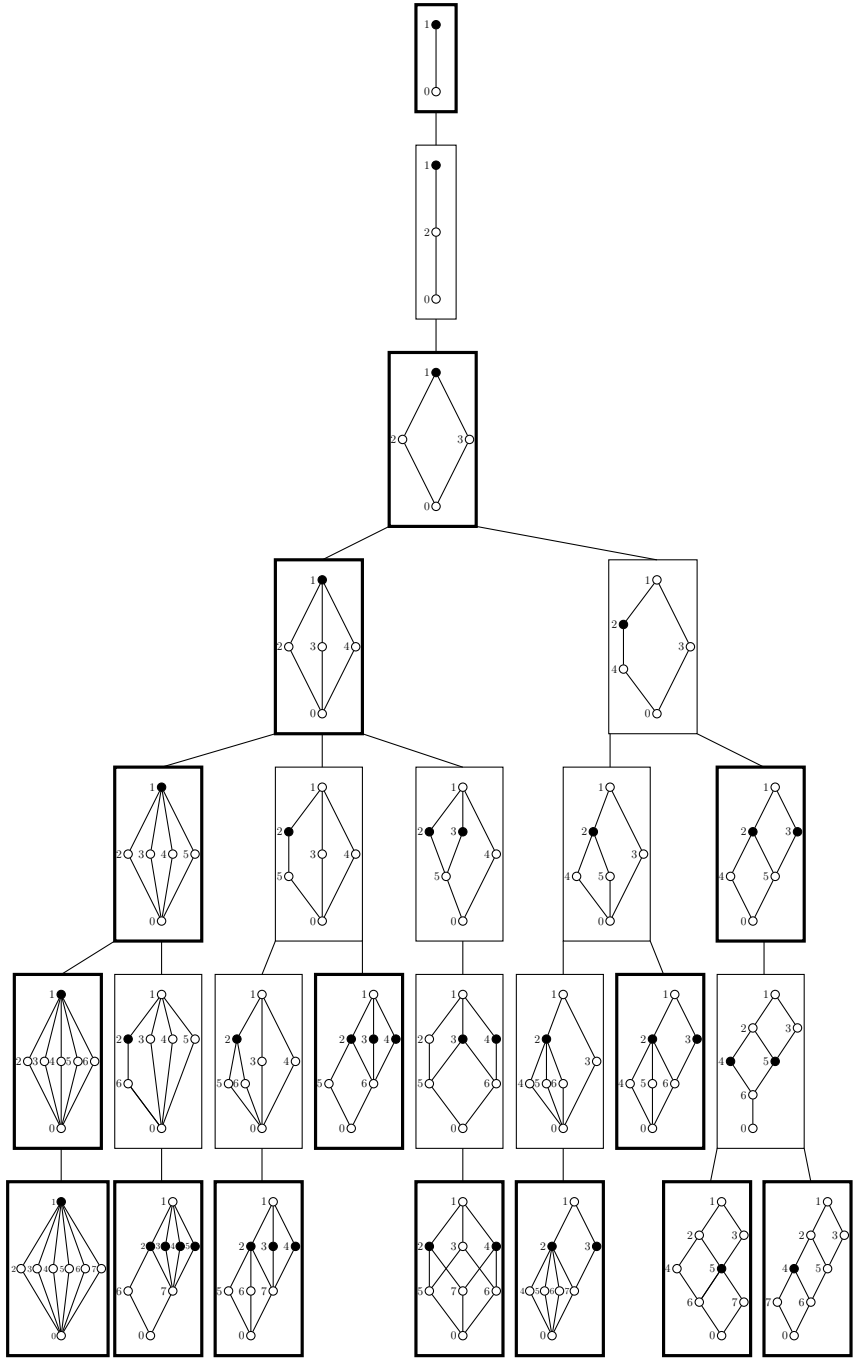
\begin{figure}
\begin{center}
\begin{tikzpicture}[scale=0.89, every node/.style={draw, circle, inner sep=2.8pt,scale=0.4}]\large
	\node (a0) at(0,0){};
	\node [fill=black](a1) at($(a0)+(0,1)$){};
	\node [scale=1,draw=none, circle=none, fill=none, left] at($(a1)$){1};
	\node [scale=1,draw=none, circle=none, fill=none, left] at($(a0)$){0};
	\draw(a0)edge(a1);
	\draw [very thick]($(a0)+(-0.3,-0.3)$) rectangle ($(a0)+(0.3,1.3)$);
	\draw [black] ($(a0)+(0,-0.3)$) -- +(0,-0.5);
	\node (b0) at($(a0)+(0,-3.1)$){};
	\node [fill=black](b1) at($(b0)+(0,2)$){};
	\node (b2) at($(b0)+(0,1)$){};
	\node [scale=1,draw=none, circle=none, fill=none, left] at($(b0)$){0};
	\node [scale=1,draw=none, circle=none, fill=none, left] at($(b1)$){1};
	\node [scale=1,draw=none, circle=none, fill=none, left] at($(b2)$){2};
	\draw(b0)edge(b2);
	\draw(b2)edge(b1);
	\draw ($(b0)+(-0.3,-0.3)$) rectangle ($(b0)+(0.3,2.3)$);
	\draw [black] ($(b0)+(0,-0.3)$) -- +(0,-0.5);
	\node (c0) at($(b0)+(0,-3.1)$){};
	\node [fill=black](c1) at($(c0)+(0,2)$){};
	\node (c2) at($(c0)+(-0.5,1)$){};
	\node (c3) at($(c0)+(0.5,1)$){};
	\node [scale=1,draw=none, circle=none, fill=none, left] at($(c0)$){0};
	\node [scale=1,draw=none, circle=none, fill=none, left] at($(c1)$){1};
	\node [scale=1,draw=none, circle=none, fill=none, left] at($(c2)$){2};
	\node [scale=1,draw=none, circle=none, fill=none, left] at($(c3)$){3};
	\draw(c0)edge(c2);
	\draw(c0)edge(c3);
	\draw(c3)edge(c1);
	\draw(c2)edge(c1);
	\draw [very thick]($(c0)+(-0.7,-0.3)$) rectangle ($(c0)+(0.6,2.3)$);
	\draw [black] ($(c0)+(-0.7,-0.3)$) -- +(-1,-0.5);
	\node (d0) at($(c0)+(-1.7,-3.1)$){};
	\node [fill=black](d1) at($(d0)+(0,2)$){};
	\node (d2) at($(d0)+(-0.5,1)$){};
	\node (d3) at($(d0)+(0,1)$){};
	\node (d4) at($(d0)+(0.5,1)$){};
	\node [scale=1,draw=none, circle=none, fill=none, left] at($(d0)$){0};
	\node [scale=1,draw=none, circle=none, fill=none, left] at($(d1)$){1};
	\node [scale=1,draw=none, circle=none, fill=none, left] at($(d2)$){2};
	\node [scale=1,draw=none, circle=none, fill=none, left] at($(d3)$){3};
	\node [scale=1,draw=none, circle=none, fill=none, left] at($(d4)$){4};
	\draw(d0)edge(d2)edge(d3)edge(d4);
	\draw(d2)edge(d1);
	\draw(d3)edge(d1);
	\draw(d4)edge(d1);
	\draw [very thick]($(d0)+(-0.7,-0.3)$) rectangle ($(d0)+(0.6,2.3)$); 
	\draw [black] ($(c0)+(0.6,-0.3)$) -- +(2.7,-0.5);
	\node (e0) at($(c0)+(3.3,-3.1)$){};
	\node (e1) at($(e0)+(0,2)$){};
	\node [fill=black](e2) at($(e0)+(-0.5,4/3)$){};
	\node (e3) at($(e0)+(0.5,1)$){};
	\node (e4) at($(e0)+(-0.5,2/3)$){};
	\node [scale=1,draw=none, circle=none, fill=none, left] at($(e0)$){0};
	\node [scale=1,draw=none, circle=none, fill=none, left] at($(e1)$){1};
	\node [scale=1,draw=none, circle=none, fill=none, left] at($(e2)$){2};
	\node [scale=1,draw=none, circle=none, fill=none, left] at($(e3)$){3};
	\node [scale=1,draw=none, circle=none, fill=none, left] at($(e4)$){4};
	\draw(e0)edge(e4)edge(e3);
	\draw(e3)edge(e1);
	\draw(e2)edge(e1)edge(e4);
	\draw ($(e0)+(-0.72,-0.3)$) rectangle ($(e0)+(0.6,2.3)$);

	\draw [black] ($(d0)+(-0.7,-0.3)$) -- +(-1.7,-0.5);
	\node (g0) at($(d0)+(-2.4,-3.1)$){};
	\node [fill=black](g1) at($(g0)+(0,2)$){};
	\node (g2) at($(g0)+(-0.5,1)$){};
	\node (g3) at($(g0)+(-0.5+1/3,1)$){};
	\node (g4) at($(g0)+(0.5-1/3,1)$){};
	\node (g5) at($(g0)+(0.5,1)$){};
	\node [scale=1,draw=none, circle=none, fill=none, left] at($(g0)$){0};
	\node [scale=1,draw=none, circle=none, fill=none, left] at($(g1)$){1};
	\node [scale=1,draw=none, circle=none, fill=none, left] at($(g2)$){2};
	\node [scale=1,draw=none, circle=none, fill=none, left] at($(g3)$){3};
	\node [scale=1,draw=none, circle=none, fill=none, left] at($(g4)$){4};
	\node [scale=1,draw=none, circle=none, fill=none, left] at($(g5)$){5};
	\draw(g0)edge(g2)edge(g3)edge(g4)edge(g5);
	\draw(g1)edge(g2)edge(g3)edge(g4)edge(g5);
	\draw [very thick]($(g0)+(-0.7,-0.3)$) rectangle ($(g0)+(0.6,2.3)$);
	\draw [black] ($(d0)+(0,-0.3)$) -- +(0,-0.5);
	\node (h0) at($(d0)+(0,-3.1)$){};
	\node (h1) at($(h0)+(0,2)$){};
	\node [fill=black](h2) at($(h0)+(-0.5,4/3)$){};
	\node (h3) at($(h0)+(0,1)$){};
	\node (h4) at($(h0)+(0.5,1)$){};
	\node (h5) at($(h0)+(-0.5,2/3)$){};
	\node [scale=1,draw=none, circle=none, fill=none, left] at($(h0)$){0};
	\node [scale=1,draw=none, circle=none, fill=none, left] at($(h1)$){1};
	\node [scale=1,draw=none, circle=none, fill=none, left] at($(h2)$){2};
	\node [scale=1,draw=none, circle=none, fill=none, left] at($(h3)$){3};
	\node [scale=1,draw=none, circle=none, fill=none, left] at($(h4)$){4};
	\node [scale=1,draw=none, circle=none, fill=none, left] at($(h5)$){5};
	\draw(h0)edge(h3)edge(h4)edge(h5);
	\draw(h1)edge(h2)edge(h3)edge(h4);
	\draw(h2)edge(h5);
	\draw ($(h0)+(-0.7,-0.3)$) rectangle ($(h0)+(0.6,2.3)$);
	\draw [black] ($(d0)+(0.6,-0.3)$) -- +(1.5,-0.5);		
	\node (i0) at($(d0)+(2.1,-3.1)$){};
	\node (i1) at($(i0)+(0,2)$){};
	\node [fill=black](i2) at($(i0)+(-0.5,4/3)$){};
	\node [fill=black](i3) at($(i0)+(0,4/3)$){};
	\node (i4) at($(i0)+(0.5,1)$){};
	\node (i5) at($(i0)+(-0.25,2/3)$){};
	\node [scale=1,draw=none, circle=none, fill=none, left] at($(i0)$){0};
	\node [scale=1,draw=none, circle=none, fill=none, left] at($(i1)$){1};
	\node [scale=1,draw=none, circle=none, fill=none, left] at($(i2)$){2};
	\node [scale=1,draw=none, circle=none, fill=none, left] at($(i3)$){3};
	\node [scale=1,draw=none, circle=none, fill=none, left] at($(i4)$){4};
	\node [scale=1,draw=none, circle=none, fill=none, left] at($(i5)$){5};
	\draw(i0)edge(i4)edge(i5);
	\draw(i1)edge(i2)edge(i3)edge(i4);
	\draw(i5)edge(i2)edge(i3);
	\draw ($(i0)+(-0.7,-0.3)$) rectangle ($(i0)+(0.6,2.3)$);
	\draw [black] ($(e0)+(0.6,-0.3)$) -- +(1,-0.5);
	\node (j0) at($(e0)+(1.6-1/6,-3.1)$){};
	\node (j1) at($(j0)+(1/3,2)$){};
	\node [fill=black](j2) at($(j0)+(0,4/3)$){};
	\node [fill=black](j3) at($(j0)+(2/3,4/3)$){};
	\node (j4) at($(j0)+(-1/3,2/3)$){};
	\node (j5) at($(j0)+(1/3,2/3)$){};
	\node [scale=1,draw=none, circle=none, fill=none, left] at($(j0)$){0};
	\node [scale=1,draw=none, circle=none, fill=none, left] at($(j1)$){1};
	\node [scale=1,draw=none, circle=none, fill=none, left] at($(j2)$){2};
	\node [scale=1,draw=none, circle=none, fill=none, left] at($(j3)$){3};
	\node [scale=1,draw=none, circle=none, fill=none, left] at($(j4)$){4};
	\node [scale=1,draw=none, circle=none, fill=none, left] at($(j5)$){5};
	\draw(j0)edge(j4)edge(j5);
	\draw(j1)edge(j2)edge(j3);
	\draw(j5)edge(j2)edge(j3);
	\draw(j4)edge(j2);
	\draw [very thick]($(j0)+(-0.7+1/6,-0.3)$) rectangle ($(j0)+(0.6+1/6,2.3)$);
	\draw [black] ($(e0)+(-0.7,-0.3)$) -- +(0,-0.5);
	\node (k0) at($(e0)+(-0.7,-3.1)$){};
	\node (k1) at($(k0)+(0,2)$){};
	\node [fill=black](k2) at($(k0)+(-0.25,4/3)$){};
	\node (k3) at($(k0)+(0.5,1)$){};
	\node (k5) at($(k0)+(0,2/3)$){};
	\node (k4) at($(k0)+(-0.5,2/3)$){};
	\node [scale=1,draw=none, circle=none, fill=none, left] at($(k0)$){0};
	\node [scale=1,draw=none, circle=none, fill=none, left] at($(k1)$){1};
	\node [scale=1,draw=none, circle=none, fill=none, left] at($(k2)$){2};
	\node [scale=1,draw=none, circle=none, fill=none, left] at($(k3)$){3};
	\node [scale=1,draw=none, circle=none, fill=none, left] at($(k4)$){4};
	\node [scale=1,draw=none, circle=none, fill=none, left] at($(k5)$){5};
	\draw(k0)edge(k4)edge(k5)edge(k3);
	\draw(k1)edge(k2)edge(k3);
	\draw(k2)edge(k4)edge(k5);
	\draw ($(k0)+(-0.7,-0.3)$) rectangle ($(k0)+(0.6,2.3)$);
	\draw [black] ($(g0)+(-0.7,-0.3)$) -- +(-0.8,-0.5);
	\node (l0) at($(g0)+(-1.5,-3.1)$){};
	\node [fill=black](l1) at($(l0)+(0,2)$){};
	\node (l2) at($(l0)+(-0.5,1)$){};
	\node (l3) at($(l0)+(-0.25,1)$){};
	\node (l4) at($(l0)+(0,1)$){};
	\node (l5) at($(l0)+(0.25,1)$){};
	\node (l6) at($(l0)+(0.5,1)$){};
	\node [scale=1,draw=none, circle=none, fill=none, left] at($(l0)$){0};
	\node [scale=1,draw=none, circle=none, fill=none, left] at($(l1)$){1};
	\node [scale=1,draw=none, circle=none, fill=none, left] at($(l2)$){2};
	\node [scale=1,draw=none, circle=none, fill=none, left] at($(l3)$){3};
	\node [scale=1,draw=none, circle=none, fill=none, left] at($(l4)$){4};
	\node [scale=1,draw=none, circle=none, fill=none, left] at($(l5)$){5};
	\node [scale=1,draw=none, circle=none, fill=none, left] at($(l6)$){6};
	\draw(l0)edge(l2)edge(l3)edge(l4)edge(l5)edge(l6);
	\draw(l1)edge(l2)edge(l3)edge(l4)edge(l5)edge(l6);
	\draw [very thick]($(l0)+(-0.7,-0.3)$) rectangle ($(l0)+(0.6,2.3)$);
	\draw [black] ($(g0)+(0,-0.3)$) -- +(0,-0.5);
	\node (m0) at($(g0)+(0,-3.1)$){};
	\node (m1) at($(m0)+(0,2)$){};
	\node [fill=black](m2) at($(m0)+(-0.5,4/3)$){};
	\node (m3) at($(m0)+(-0.5+1/3,4/3)$){};
	\node (m4) at($(m0)+(0.5-1/3,4/3)$){};
	\node (m5) at($(m0)+(0.5,4/3)$){};
	\node (m6) at($(m0)+(-0.5,2/3)$){};
	\node [scale=1,draw=none, circle=none, fill=none, left] at($(m0)$){0};
	\node [scale=1,draw=none, circle=none, fill=none, left] at($(m1)$){1};
	\node [scale=1,draw=none, circle=none, fill=none, left] at($(m2)$){2};
	\node [scale=1,draw=none, circle=none, fill=none, left] at($(m3)$){3};
	\node [scale=1,draw=none, circle=none, fill=none, left] at($(m4)$){4};
	\node [scale=1,draw=none, circle=none, fill=none, left] at($(m5)$){5};
	\node [scale=1,draw=none, circle=none, fill=none, left] at($(m6)$){6};
	\draw(m0)edge(m6)edge(m3)edge(m4)edge(m5)edge(m6);
	\draw(m1)edge(m2)edge(m3)edge(m4)edge(m5);
	\draw(m2)edge(m6);
	\draw ($(m0)+(-0.7,-0.3)$) rectangle ($(m0)+(0.6,2.3)$);
	\draw [black] ($(h0)+(-0.7,-0.3)$) -- +(-0.2,-0.5);		
	\node (n0) at($(h0)+(-0.9,-3.1)$){};
	\node (n1) at($(n0)+(0,2)$){};
	\node [fill=black](n2) at($(n0)+(-0.375,4/3)$){};
	\node (n3) at($(n0)+(0,1)$){};
	\node (n4) at($(n0)+(0.5,1)$){};
	\node (n5) at($(n0)+(-0.5,2/3)$){};
	\node (n6) at($(n0)+(-0.25,2/3)$){};
	\node [scale=1,draw=none, circle=none, fill=none, left] at($(n0)$){0};
	\node [scale=1,draw=none, circle=none, fill=none, left] at($(n1)$){1};
	\node [scale=1,draw=none, circle=none, fill=none, left] at($(n2)$){2};
	\node [scale=1,draw=none, circle=none, fill=none, left] at($(n3)$){3};
	\node [scale=1,draw=none, circle=none, fill=none, left] at($(n4)$){4};
	\node [scale=1,draw=none, circle=none, fill=none, left] at($(n5)$){5};
	\node [scale=1,draw=none, circle=none, fill=none, left] at($(n6)$){6};
	\draw(n0)edge(n3)edge(n4)edge(n5)edge(n6);
	\draw(n1)edge(n2)edge(n3)edge(n4);
	\draw(n2)edge(n5)edge(n6);
	\draw ($(n0)+(-0.7,-0.3)$) rectangle ($(n0)+(0.6,2.3)$);
	\draw [black] ($(h0)+(0.6,-0.3)$) -- +(0,-0.5);
	\node (o0) at($(h0)+(0.6-1/6,-3.1)$){};
	\node (o1) at($(o0)+(1/3,2)$){};
	\node [fill=black](o2) at($(o0)+(0,4/3)$){};
	\node [fill=black](o3) at($(o0)+(1/3,4/3)$){};
	\node [fill=black](o4) at($(o0)+(2/3,4/3)$){};
	\node (o5) at($(o0)+(-1/3,2/3)$){};
	\node (o6) at($(o0)+(1/3,2/3)$){};
	\node [scale=1,draw=none, circle=none, fill=none, left] at($(o0)$){0};
	\node [scale=1,draw=none, circle=none, fill=none, left] at($(o1)$){1};
	\node [scale=1,draw=none, circle=none, fill=none, left] at($(o2)$){2};
	\node [scale=1,draw=none, circle=none, fill=none, left] at($(o3)$){3};
	\node [scale=1,draw=none, circle=none, fill=none, left] at($(o4)$){4};
	\node [scale=1,draw=none, circle=none, fill=none, left] at($(o5)$){5};
	\node [scale=1,draw=none, circle=none, fill=none, left] at($(o6)$){6};
	\draw(o0)edge(o5)edge(o6);
	\draw(o1)edge(o2)edge(o3)edge(o4);
	\draw(o5)edge(o2);
	\draw(o6)edge(o2)edge(o3)edge(o4);
	\draw [very thick]($(o0)+(-0.7+1/6,-0.3)$) rectangle ($(o0)+(0.6+1/6,2.3)$);
	\draw [black] ($(i0)+(0,-0.3)$) -- +(0,-0.5);
	\node (p0) at($(i0)+(0,-3.1)$){};
	\node (p1) at($(p0)+(0,2)$){};
	\node (p2) at($(p0)+(-0.5,4/3)$){};
	\node [fill=black](p3) at($(p0)+(0,4/3)$){};
	\node [fill=black](p4) at($(p0)+(0.5,4/3)$){};
	\node (p5) at($(p0)+(-0.5,2/3)$){};
	\node (p6) at($(p0)+(0.5,2/3)$){};
	\node [scale=1,draw=none, circle=none, fill=none, left] at($(p0)$){0};
	\node [scale=1,draw=none, circle=none, fill=none, left] at($(p1)$){1};
	\node [scale=1,draw=none, circle=none, fill=none, left] at($(p2)$){2};
	\node [scale=1,draw=none, circle=none, fill=none, left] at($(p3)$){3};
	\node [scale=1,draw=none, circle=none, fill=none, left] at($(p4)$){4};
	\node [scale=1,draw=none, circle=none, fill=none, left] at($(p5)$){5};
	\node [scale=1,draw=none, circle=none, fill=none, left] at($(p6)$){6};
	\draw(p0)edge(p5)edge(p6);
	\draw(p1)edge(p2)edge(p3)edge(p4);
	\draw(p2)edge(p5);
	\draw(p4)edge(p6);
	\draw(p3)edge(p5)edge(p6);
	\draw ($(p0)+(-0.7,-0.3)$) rectangle ($(p0)+(0.6,2.3)$);
	\draw [black] ($(j0)+(1/6,-0.3)$) -- +(0,-0.5);
	\node (q0) at($(j0)+(0,-3.1)$){};
	\node (q1) at($(q0)+(1/3,2)$){};
	\node (q2) at($(q0)+(0,1.5)$){};
	\node (q3) at($(q0)+(2/3,1.5)$){};
	\node [fill=black](q4) at($(q0)+(-1/3,1)$){};
	\node [fill=black](q5) at($(q0)+(1/3,1)$){};
	\node (q6) at($(q0)+(0,0.5)$){};
	\node [scale=1,draw=none, circle=none, fill=none, left] at($(q0)$){0};
	\node [scale=1,draw=none, circle=none, fill=none, left] at($(q1)$){1};
	\node [scale=1,draw=none, circle=none, fill=none, left] at($(q2)$){2};
	\node [scale=1,draw=none, circle=none, fill=none, left] at($(q3)$){3};
	\node [scale=1,draw=none, circle=none, fill=none, left] at($(q4)$){4};
	\node [scale=1,draw=none, circle=none, fill=none, left] at($(q5)$){5};
	\node [scale=1,draw=none, circle=none, fill=none, left] at($(q6)$){6};
	\draw(q6)edge(q4)edge(q5)edge(q0);
	\draw(q1)edge(q2)edge(q3);
	\draw(q5)edge(q2)edge(q3);
	\draw(q4)edge(q2);
	\draw ($(q0)+(-0.7+1/6,-0.3)$) rectangle ($(q0)+(0.6+1/6,2.3)$);
	\draw [black] ($(k0)+(-0.7,-0.3)$) -- +(0,-0.5);
	\node (r0) at($(k0)+(-0.7,-3.1)$){};
	\node (r1) at($(r0)+(0,2)$){};
	\node [fill=black](r2) at($(r0)+(-0.25,4/3)$){};
	\node (r3) at($(r0)+(0.5,1)$){};
	\node (r6) at($(r0)+(0,2/3)$){};
	\node (r5) at($(r0)+(-0.25,2/3)$){};
	\node (r4) at($(r0)+(-0.5,2/3)$){};
	\node [scale=1,draw=none, circle=none, fill=none, left] at($(r0)$){0};
	\node [scale=1,draw=none, circle=none, fill=none, left] at($(r1)$){1};
	\node [scale=1,draw=none, circle=none, fill=none, left] at($(r2)$){2};
	\node [scale=1,draw=none, circle=none, fill=none, left] at($(r3)$){3};
	\node [scale=1,draw=none, circle=none, fill=none, left] at($(r4)$){4};
	\node [scale=1,draw=none, circle=none, fill=none, left] at($(r5)$){5};
	\node [scale=1,draw=none, circle=none, fill=none, left] at($(r6)$){6};
	\draw(r0)edge(r4)edge(r5)edge(r3)edge(r6);
	\draw(r1)edge(r2)edge(r3);
	\draw(r2)edge(r4)edge(r5)edge(r6);
	\draw ($(r0)+(-0.7,-0.3)$) rectangle ($(r0)+(0.6,2.3)$);
	\draw [black] ($(k0)+(0.6,-0.3)$) -- +(0.2,-0.5);		
	\node (s0) at($(k0)+(0.8-1/6,-3.1)$){};
	\node (s1) at($(s0)+(1/3,2)$){};
	\node [fill=black](s2) at($(s0)+(0,4/3)$){};
	\node [fill=black](s3) at($(s0)+(2/3,4/3)$){};
	\node (s4) at($(s0)+(-1/3,2/3)$){};
	\node (s5) at($(s0)+(0,2/3)$){};
	\node (s6) at($(s0)+(1/3,2/3)$){};
	\node [scale=1,draw=none, circle=none, fill=none, left] at($(s0)$){0};
	\node [scale=1,draw=none, circle=none, fill=none, left] at($(s1)$){1};
	\node [scale=1,draw=none, circle=none, fill=none, left] at($(s2)$){2};
	\node [scale=1,draw=none, circle=none, fill=none, left] at($(s3)$){3};
	\node [scale=1,draw=none, circle=none, fill=none, left] at($(s4)$){4};
	\node [scale=1,draw=none, circle=none, fill=none, left] at($(s5)$){5};
	\node [scale=1,draw=none, circle=none, fill=none, left] at($(s6)$){6};
	\draw(s2)edge(s1)edge(s4)edge(s5)edge(s6);
	\draw(s0)edge(s4)edge(s5)edge(s6);
	\draw(s3)edge(s1)edge(s6);
	\draw [very thick]($(s0)+(-0.7+1/6,-0.3)$) rectangle ($(s0)+(0.6+1/6,2.3)$);
	
	\draw [black] ($(l0)+(0,-0.3)$) -- +(0,-0.5);
	\node (t0) at($(l0)+(0,-3.1)$){};
	\node [fill=black](t1) at($(t0)+(0,2)$){};
	\node (t2) at($(t0)+(-0.6,1)$){};
	\node (t3) at($(t0)+(-0.36,1)$){};
	\node (t4) at($(t0)+(-0.12,1)$){};
	\node (t5) at($(t0)+(0.12,1)$){};
	\node (t6) at($(t0)+(0.36,1)$){};
	\node (t7) at($(t0)+(0.6,1)$){};
	\node [scale=0.8,draw=none, circle=none, fill=none, left] at($(t0)$){0};
	\node [scale=0.8,draw=none, circle=none, fill=none, left] at($(t1)$){1};
	\node [scale=0.8,draw=none, circle=none, fill=none, left] at($(t2)$){2};
	\node [scale=0.8,draw=none, circle=none, fill=none, left] at($(t3)$){3};
	\node [scale=0.8,draw=none, circle=none, fill=none, left] at($(t4)$){4};
	\node [scale=0.8,draw=none, circle=none, fill=none, left] at($(t5)$){5};
	\node [scale=0.8,draw=none, circle=none, fill=none, left] at($(t6)$){6};
	\node [scale=0.8,draw=none, circle=none, fill=none, left] at($(t7)$){7};
	\draw(t0)edge(t2)edge(t3)edge(t4)edge(t5)edge(t6)edge(t7);
	\draw(t1)edge(t2)edge(t3)edge(t4)edge(t5)edge(t6)edge(t7);
	\draw [very thick]($(t0)+(-0.8,-0.3)$) rectangle ($(t0)+(0.7,2.3)$);
	\draw [black] ($(m0)+(0,-0.3)$) -- +(0,-0.5);
	\node (u0) at($(m0)+(-1/6,-3.1)$){};
	\node (u1) at($(u0)+(1/3,2)$){};
	\node [fill=black](u2) at($(u0)+(0,4/3)$){};
	\node [fill=black](u3) at($(u0)+(2/9,4/3)$){};
	\node [fill=black](u4) at($(u0)+(4/9,4/3)$){};
	\node [fill=black](u5) at($(u0)+(2/3,4/3)$){};
	\node (u6) at($(u0)+(-1/3,2/3)$){};
	\node (u7) at($(u0)+(1/3,2/3)$){};
	\node [scale=1,draw=none, circle=none, fill=none, left] at($(u0)$){0};
	\node [scale=1,draw=none, circle=none, fill=none, left] at($(u1)$){1};
	\node [scale=0.8,draw=none, circle=none, fill=none, left] at($(u2)$){2};
	\node [scale=0.8,draw=none, circle=none, fill=none, left] at($(u3)$){3};
	\node [scale=0.8,draw=none, circle=none, fill=none, left] at($(u4)$){4};
	\node [scale=0.8,draw=none, circle=none, fill=none, left] at($(u5)$){5};
	\node [scale=1,draw=none, circle=none, fill=none, left] at($(u6)$){6};
	\node [scale=1,draw=none, circle=none, fill=none, left] at($(u7)$){7};
	\draw(u0)edge(u6)edge(u7);
	\draw(u1)edge(u2)edge(u3)edge(u4)edge(u5);
	\draw(u7)edge(u2)edge(u3)edge(u4)edge(u5);
	\draw(u6)edge(u2);
	\draw [very thick]($(u0)+(-0.7+1/6,-0.3)$) rectangle ($(u0)+(0.6+1/6,2.3)$);
	\draw [black] ($(n0)+(0,-0.3)$) -- +(0,-0.5);
	\node (v0) at($(n0)+(-1/6,-3.1)$){};
	\node (v1) at($(v0)+(1/3,2)$){};
	\node [fill=black](v2) at($(v0)+(0,4/3)$){};
	\node [fill=black](v3) at($(v0)+(1/3,4/3)$){};
	\node [fill=black](v4) at($(v0)+(2/3,4/3)$){};
	\node (v5) at($(v0)+(-1/3,2/3)$){};
	\node (v6) at($(v0)+(0,2/3)$){};
	\node (v7) at($(v0)+(1/3,2/3)$){};
	\node [scale=1,draw=none, circle=none, fill=none, left] at($(v0)$){0};
	\node [scale=1,draw=none, circle=none, fill=none, left] at($(v1)$){1};
	\node [scale=1,draw=none, circle=none, fill=none, left] at($(v2)$){2};
	\node [scale=1,draw=none, circle=none, fill=none, left] at($(v3)$){3};
	\node [scale=1,draw=none, circle=none, fill=none, left] at($(v4)$){4};
	\node [scale=1,draw=none, circle=none, fill=none, left] at($(v5)$){5};
	\node [scale=1,draw=none, circle=none, fill=none, left] at($(v6)$){6};
	\node [scale=1,draw=none, circle=none, fill=none, left] at($(v7)$){7};
	\draw(v0)edge(v5)edge(v6)edge(v7);
	\draw(v1)edge(v2)edge(v3)edge(v4);
	\draw(v7)edge(v2)edge(v3)edge(v4);
	\draw(v2)edge(v5)edge(v6);
	\draw [very thick]($(v0)+(-0.7+1/6,-0.3)$) rectangle ($(v0)+(0.6+1/6,2.3)$);
	\draw [black] ($(p0)+(0,-0.3)$) -- +(0,-0.5);
	\node (w0) at($(p0)+(0,-3.1)$){};
	\node (w1) at($(w0)+(0,2)$){};
	\node [fill=black](w2) at($(w0)+(-0.5,4/3)$){};
	\node (w3) at($(w0)+(0,4/3)$){};
	\node [fill=black](w4) at($(w0)+(0.5,4/3)$){};
	\node (w5) at($(w0)+(-0.5,2/3)$){};
	\node (w6) at($(w0)+(0.5,2/3)$){};
	\node (w7) at($(w0)+(0,2/3)$){};
	\node [scale=1,draw=none, circle=none, fill=none, left] at($(w0)$){0};
	\node [scale=1,draw=none, circle=none, fill=none, left] at($(w1)$){1};
	\node [scale=1,draw=none, circle=none, fill=none, left] at($(w2)$){2};
	\node [scale=1,draw=none, circle=none, fill=none, left] at($(w3)$){3};
	\node [scale=1,draw=none, circle=none, fill=none, left] at($(w4)$){4};
	\node [scale=1,draw=none, circle=none, fill=none, left] at($(w5)$){5};
	\node [scale=1,draw=none, circle=none, fill=none, left] at($(w6)$){6};
	\node [scale=1,draw=none, circle=none, fill=none, left] at($(w7)$){7};
	\draw(w0)edge(w5)edge(w6);
	\draw(w1)edge(w2)edge(w3)edge(w4);
	\draw(w2)edge(w5);
	\draw(w4)edge(w6);
	\draw(w3)edge(w5)edge(w6);
	\draw(w7)edge(w0)edge(w2)edge(w4);
	\draw [very thick]($(w0)+(-0.7,-0.3)$) rectangle ($(w0)+(0.6,2.3)$);
	\draw [black] ($(q0)+(0.6+1/6,-0.3)$) -- +(0.1,-0.5);	
	\node (x0) at($(q0)+(0.7+1/6-0.25,-3.1)$){};
	\node (x1) at($(x0)+(0.5,2)$){};
	\node (x2) at($(x0)+(0.25,1.5)$){};
	\node (x3) at($(x0)+(0.75,1.5)$){};
	\node [fill=black](x4) at($(x0)+(0,1)$){};
	\node (x5) at($(x0)+(0.5,1)$){};
	\node (x6) at($(x0)+(0.25,0.5)$){};
	\node (x7) at($(x0)+(-0.25,0.5)$){};
	\node [scale=1,draw=none, circle=none, fill=none, left] at($(x0)$){0};
	\node [scale=1,draw=none, circle=none, fill=none, left] at($(x1)$){1};
	\node [scale=1,draw=none, circle=none, fill=none, left] at($(x2)$){2};
	\node [scale=1,draw=none, circle=none, fill=none, left] at($(x3)$){3};
	\node [scale=1,draw=none, circle=none, fill=none, left] at($(x4)$){4};
	\node [scale=1,draw=none, circle=none, fill=none, left] at($(x5)$){5};
	\node [scale=1,draw=none, circle=none, fill=none, left] at($(x6)$){6};
	\node [scale=1,draw=none, circle=none, fill=none, left] at($(x7)$){7};
	\draw(x6)edge(x4)edge(x5)edge(x0);
	\draw(x1)edge(x2)edge(x3);
	\draw(x5)edge(x2)edge(x3);
	\draw(x4)edge(x2);
	\draw(x7)edge(x0)edge(x4);
	\draw [very thick]($(x0)+(-0.7+0.25,-0.3)$) rectangle ($(x0)+(0.6+0.25,2.3)$);
	\draw [black] ($(q0)+(-0.7+1/6,-0.3)$) -- +(-0.1,-0.5);		
	\node (z0) at($(q0)+(-0.8+1/3,-3.1)$){};
	\node (z1) at($(z0)+(0,2)$){};
	\node (z2) at($(z0)+(-1/3,1.5)$){};
	\node (z3) at($(z0)+(1/3,1.5)$){};
	\node (z4) at($(z0)+(-2/3,1)$){};
	\node [fill=black](z5) at($(z0)+(0,1)$){};
	\node (z6) at($(z0)+(-1/3,0.5)$){};
	\node (z7) at($(z0)+(1/3,0.5)$){};
	\node [scale=1,draw=none, circle=none, fill=none, left] at($(z0)$){0};
	\node [scale=1,draw=none, circle=none, fill=none, left] at($(z1)$){1};
	\node [scale=1,draw=none, circle=none, fill=none, left] at($(z2)$){2};
	\node [scale=1,draw=none, circle=none, fill=none, left] at($(z3)$){3};
	\node [scale=1,draw=none, circle=none, fill=none, left] at($(z4)$){4};
	\node [scale=1,draw=none, circle=none, fill=none, left] at($(z5)$){5};
	\node [scale=1,draw=none, circle=none, fill=none, left] at($(z6)$){6};
	\node [scale=1,draw=none, circle=none, fill=none, left] at($(z7)$){7};
	\draw(z6)edge(z4)edge(z5)edge(z0);
	\draw(z1)edge(z2)edge(z3);
	\draw(z5)edge(z2)edge(z3)edge(z6)edge(z7);
	\draw(z4)edge(z2);
	\draw(z7)edge(z0);
	\draw [very thick]($(z0)+(-0.7-1/6,-0.3)$) rectangle ($(z0)+(0.6-1/6,2.3)$);
	\draw [black] ($(r0)+(0,-0.3)$) -- +(0,-0.5);
	\node (y0) at($(r0)+(-1/6,-3.1)$){};
	\node (y1) at($(y0)+(1/3,2)$){};
	\node [fill=black](y2) at($(y0)+(0,4/3)$){};
	\node [fill=black](y3) at($(y0)+(2/3,4/3)$){};
	\node (y4) at($(y0)+(-1/3,2/3)$){};
	\node (y5) at($(y0)+(-1/9,2/3)$){};
	\node (y6) at($(y0)+(1/9,2/3)$){};
	\node (y7) at($(y0)+(1/3,2/3)$){};
	\node [scale=1,draw=none, circle=none, fill=none, left] at($(y0)$){0};
	\node [scale=1,draw=none, circle=none, fill=none, left] at($(y1)$){1};
	\node [scale=1,draw=none, circle=none, fill=none, left] at($(y2)$){2};
	\node [scale=1,draw=none, circle=none, fill=none, left] at($(y3)$){3};
	\node [scale=0.8,draw=none, circle=none, fill=none, left] at($(y4)$){4};
	\node [scale=0.8,draw=none, circle=none, fill=none, left] at($(y5)$){5};
	\node [scale=0.8,draw=none, circle=none, fill=none, left] at($(y6)$){6};
	\node [scale=0.8,draw=none, circle=none, fill=none, left] at($(y7)$){7};
	\draw(y2)edge(y1)edge(y4)edge(y5)edge(y6)edge(y7);
	\draw(y0)edge(y4)edge(y5)edge(y6)edge(y7);
	\draw(y3)edge(y1)edge(y7);
	\draw [very thick]($(y0)+(-0.7+1/6,-0.3)$) rectangle ($(y0)+(0.6+1/6,2.3)$);
\end{tikzpicture}
\caption{Example of the generation of all vertically indecomposable modular lattices up to size $n=8$. Black circles indicate lattice antichain used in the previous step. 
Thick rectangles indicate vertically indecomposable modular lattices.}
\end{center}
\end{figure}

\begin{table}[h]\scriptsize
\begin{center}
\begin{tabular}{|r|r|r|r|r|r|}\hline
$n$ & All Lattices\quad\  & Semimodular & V.I. Semimod. & Modular & V.I. Modular\\[0.5ex]\hline
1 & 1 & 1 & 1 & 1 & 1 \\
2 & 1 & 1 & 1 & 1 & 1 \\
3 & 1 & 1 & 0 & 1 & 0 \\
4 & 2 & 2 & 1 & 2 & 1 \\
5 & 5 & 4 & 1 & 4 & 1 \\
6 & 15 & 8 & 2 & 8 & 2 \\
7 & 53 &\textbf{17}&\textbf{4}& 16 & \textbf{3}\\
8 & 222 &\textbf{38}&\textbf{9}& 34 & \textbf{7}\\
9 & 1\,078 &\textbf{88}&\textbf{21}& 72 & \textbf{12}\\
10 & 5\,994 &\textbf{212}&\textbf{53}& 157 & \textbf{28}\\
11 & 37\,622 &\textbf{530}&\textbf{139}& 343 & \textbf{54}\\
12 & 262\,776 &\textbf{1\,376}&\textbf{384}& 766 & \textbf{127}\\
13 & 2\,018\,305 &\textbf{3\,693}&\textbf{1\,088}& \textbf{1\,718} & \textbf{266}\\
14 & 16\,873\,364 &\textbf{10\,232}&\textbf{3\,186}& \textbf{3\,899} & \textbf{614}\\
15 & 152\,233\,518 &\textbf{29\,231}&\textbf{9\,596}& \textbf{8\,898} & \textbf{1\,356}\\
16 & 1\,471\,613\,387&\textbf{85\,906}&\textbf{29\,601}& \textbf{20\,475} & \textbf{3\,134} \\
17 & 15\,150\,569\,446&\textbf{259\,291}&\textbf{93\,462}& \textbf{47\,321} & \textbf{7\,091}\\
18 & 165\,269\,824\,761&\textbf{802\,308}&\textbf{301\,265}& \textbf{110\,024} & \textbf{16\,482}\\
19 & \textbf{1901910625578} &\textbf{2\,540\,635}&\textbf{990\,083}& \textbf{256\,791} & \textbf{37\,929} \\
20 & &\textbf{8\,220\,218}&\textbf{3\,312\,563}& \textbf{601\,991} & \textbf{88\,622}\\
21 & &\textbf{27\,134\,483}&\textbf{11\,270\,507}& \textbf{1\,415\,768} & \textbf{206\,295}\\
22 & &\textbf{91\,258\,141}&\textbf{38\,955\,164}& \textbf{3\,340\,847} & \textbf{484\,445}\\
23 & & & & \textbf{7\,904\,700} &\textbf{1\,136\,897}\\
24 & & & & \textbf{18\,752\,942} &\textbf{2\,682\,450}\\\hline
\end{tabular} 
\end{center}
\caption{Number of lattices and (vertically indecomposable = V.I.) (semi)modular lattices up to isomorphism. New numbers are in bold.}\label{table1}
\end{table}
 
\section{A lower bound on the number of modular lattices}
Let $m_n$ denote the number of modular lattices of size $n$ (up to isomorphism).
In this section we give a simple argument for a lower bound of this sequence.

\begin{theorem} For all $n$, \ $2^{n-3}\leq m_n$.
\end{theorem}
\begin{proof}
Let $L_3$ be the three element lattice with 0 and 1 as bottom and top
respectively. Consider the following two extensions of an $n$-lattice
$L$:
$$L_\alpha:=L^A \text{ where } A=\{x\in L\mid x\succ 0\}$$
$$L_\beta:=L^{\{b\}} \text{ for an arbitrary } b \text{ such that }
b\succ n-1$$ 
We declare an $n$-lattice $L$ to be an $\alpha$-lattice or
a $\beta$-lattice if it is obtained through the $\alpha$ or $\beta$
construction respectively.
 	
We want to show that, starting with $L_3$, in each step
this construction generates
two more modular lattices that are nonisomorphic to all
other lattices in the collection.
 	
It is clear that if $L$ is a modular lattice, $L_\alpha$ is also
modular since it is the same lattice with an element added at the bottom.
 	
For $L_\beta$, we consider the cases where $L$ is a modular $\alpha$-lattice and
a modular $\beta$-lattice obtained through this
construction.  If $L$ is an $\alpha$-lattice, then there are two atoms
$n$ and $n-1$ in $L_\beta$, both of which are covered by $b$,
independently of the choice of $b$.  Therefore, $n$ and $n-1$ satisfy
(lower) semimodularity.  Since the new element $n$ is not the common
cover or co-cover of any two elements in $L_\beta$ and $L$ is modular,
it follows that $L_\beta$ is modular.

If $L$ is a $\beta$-lattice, then there is only one choice of $b$ (the
element used in the previous step) since there is only one cover for
$n-1$.  Notice that the first $\beta$ step starting from an
$\alpha$-lattice (or $L_3$) will generate an $M_2$ sublattice formed
by the bottom element $0$, the two atoms $n$ and $n-1$ and the cover
$b$ used.  After $k$ successive $\beta$ steps, there will be an
$M_{k+1}$ sublattice formed by $0$, $b$, and the atoms
$n-k$,\ldots,$n$.  Therefore, after any $\beta$ step, the new element
will share a common cover ($b$) with all the atoms it shares a common
co-cover ($0$) with, and vice versa.  Since all other elements of
$L_\beta$ satisfy (lower) semimodularity by modularity of $L$, it
follows that $L_\beta$ is modular.

Next, we want to show that $L_\alpha$ and $L_\beta$ are not isomorphic to
any other lattice obtained from this construction. Consider two modular lattices
$L$ and $M$ generated by the $\alpha$-$\beta$ construction.
Since $L_\alpha$ has a unique atom and $M_\beta$ has more than one atom,
it follows that $L_\alpha\not\cong M_\beta$.

Suppose $L_\alpha \cong M_\alpha$. Then, $L$ and $M$ can be
reconstructed by removing an atom in $L_\alpha$ and $M_\alpha$
respectively. Clearly, $L\cong M$, since there is only one atom to be
removed.

If $L_\beta \cong M_\beta$, there is more than one choice of atom.
Let $k+1$ be the number of atoms in both, then they both have
been obtained through $k$ $\beta$ steps, and the $k$ atoms added
through $\beta$ steps are automorphic. Therefore removal of any of
these atoms in $L_\beta$ and $M_\beta$ will generate two isomorphic
modular lattices, hence $L \cong M$. Consequently no two
non-isomorphic lattices can generate isomorphic lattices through the
$\alpha$-$\beta$ construction.
    
We conclude the proof by induction. For $n=3$, there are $2^{3-3}=1$
modular lattices in the $\alpha-\beta$ construction (the initial $L_3$
lattice). Assume there are $2^{n-3}$ non-isomorphic modular lattices of
size $n$ constructed via the $\alpha$-$\beta$ construction. Then
each of these lattices will produce 2 new modular lattices which are
not isomorphic to any of the lattices produced by any other non-isomorphic
lattice. As a result, there are $2^{n+1-3}$ non-isomorphic modular
lattice of size $n+1$, thus completing the induction.
 	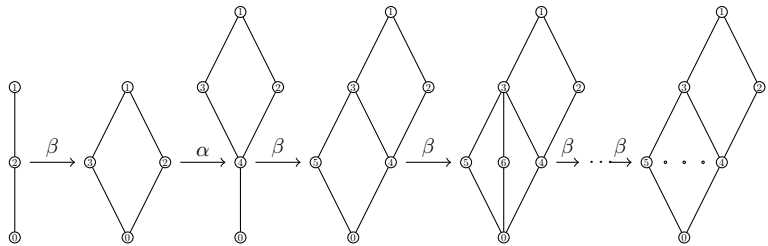
\begin{figure}
 	\begin{tikzpicture}[every node/.style={draw, circle, inner sep=0.8pt,scale=0.4}]
	\node (a0) at(0,0){0};
	\node (a2) at($(a0)+(0,1)$){2};
	\node (a1) at($(a0)+(0,2)$){1};
	\draw(a2)edge(a1)edge(a0);
	\draw [->,black] ($(a0)+(0.2,1)$) -- +(0.6,0);
	\node (b0) at($(a0)+(1.5,0)$){0};
	\node (b2) at($(b0)+(0.5,1)$){2};
	\node (b3) at($(b0)+(-0.5,1)$){3};
	\node (b1) at($(b0)+(0,2)$){1};
		\node [scale=2,draw=none, circle=none, fill=none, above] at($(a0)+(0.5,1)$){$\beta$};
	\draw(b2)edge(b1)edge(b0);
	\draw(b3)edge(b1)edge(b0);
	
	\draw [->,black] ($(b0)+(0.7,1)$) -- +(0.6,0);
	\node (c0) at($(b0)+(1.5,0)$){0};
	\node (c2) at($(c0)+(0.5,2)$){2};
	\node (c3) at($(c0)+(-0.5,2)$){3};
	\node (c4) at($(c0)+(0,1)$){4};
	\node (c1) at($(c0)+(0,3)$){1};
	\node [scale=2,draw=none, circle=none, fill=none, above] at($(b0)+(1,1)$){$\alpha$};
	\draw(c2)edge(c1)edge(c4);
	\draw(c3)edge(c1)edge(c4);
	\draw(c4)edge(c0);
	
	\draw [->,black] ($(c0)+(0.2,1)$) -- +(0.6,0);
	\node (d0) at($(c0)+(1.5,0)$){0};
	\node (d2) at($(d0)+(1,2)$){2};
	\node (d3) at($(d0)+(0,2)$){3};
	\node (d4) at($(d0)+(0.5,1)$){4};
	\node (d5) at($(d0)+(-0.5,1)$){5};
	\node (d1) at($(d0)+(0.5,3)$){1};
	\node [scale=2,draw=none, circle=none, fill=none, above] at($(c0)+(0.5,1)$){$\beta$};
	\draw(d2)edge(d1)edge(d4);
	\draw(d3)edge(d1)edge(d4);
	\draw(d4)edge(d0);
	\draw(d5)edge(d0)edge(d3);
	
	\draw [->,black] ($(d0)+(0.7,1)$) -- +(0.6,0);
	\node (e0) at($(d0)+(2,0)$){0};
	\node (e2) at($(e0)+(1,2)$){2};
	\node (e3) at($(e0)+(0,2)$){3};
	\node (e4) at($(e0)+(0.5,1)$){4};
	\node (e5) at($(e0)+(0,1)$){6};
	\node (e6) at($(e0)+(-0.5,1)$){5};
	\node (e1) at($(e0)+(0.5,3)$){1};
	\node [scale=2,draw=none, circle=none, fill=none, above] at($(d0)+(1,1)$){$\beta$};
	\draw(e2)edge(e1)edge(e4);
	\draw(e3)edge(e1)edge(e4);
	\draw(e4)edge(e0);
	\draw(e5)edge(e0)edge(e3);
	\draw(e6)edge(e0)edge(e3);
	
	\draw [->,black] ($(e0)+(0.7,1)$) -- +(0.3,0);
	\node [scale=2,draw=none, circle=none, fill=none,right] at($(e0)+(1.1,1)$){\ldots};
		\draw [->,black] ($(e0)+(1.4,1)$) -- +(0.3,0);
	\node (f0) at($(e0)+(2.4,0)$){0};
	\node (f2) at($(f0)+(1,2)$){2};
	\node (f3) at($(f0)+(0,2)$){3};
	\node (f4) at($(f0)+(0.5,1)$){4};
	\node (f5) at($(f0)+(-0.25,1)$){};
	\node (f6) at($(f0)+(-0.5,1)$){5};
	\node (f7) at($(f0)+(0,1)$){};
	\node (f8) at($(f0)+(0.25,1)$){};
	\node (f1) at($(f0)+(0.5,3)$){1};
	\node [scale=2,draw=none, circle=none, fill=none, above] at($(e0)+(0.85,1)$){$\beta$};
	\node [scale=2,draw=none, circle=none, fill=none, above] at($(e0)+(1.55,1)$){$\beta$};
	\draw(f2)edge(f1)edge(f4);
	\draw(f3)edge(f1)edge(f4);
	\draw(f4)edge(f0);
	\draw(f6)edge(f0)edge(f3);
	\end{tikzpicture}
\caption{Example of a particular path of the $\alpha-\beta$
  construction. Note that the $\alpha$ construction adds a single
  join-irreducible atom, and the $\beta$ construction adds an extra
  element to the lower $M_k$ sublattice.}
\end{figure}
\end{proof}

Let $l_n$ and $d_n$ be the number of (nonisomorphic) 
lattices and distributive lattices of size $n$.
Lower and upper bounds for these sequences are given in \cite{KW80}
and \cite{EHR02} respectively:
$$
(2^{\sqrt2/4})^{(n-2)^{3/2}+o((n-2)^{3/2})}< l_n< 6.11343^{(n-2)^{3/2}}
$$
$$
1.81^{n-4}<d_n<2.46^{n-1}.
$$ 
The lower bound for modular lattices obtained in the preceding theorem 
can be improved slightly for $n>7$ by
counting a larger class of planar modular lattices. However, it seems that
currently the best known upper bound for (semi)modular lattices is the
same as the one for all lattices.

As suggested by one of the referees, we conclude with some
observations about possible future research. There are several
alternative representations for (finite) modular lattices, based on
partial order geometries (see e.g. \cite{BC85, FH81}) or join-covers
or the incidence of join and meet irreducibles. It is possible that
enumeration algorithms using these representations would be more
efficient, but this has not (yet) been explored. The algorithm we use
can also be fairly easily adapted to other classes of lattices, such
as 2-distributive lattices or lattices of breadth $\le 2$, either with
or without adding modularity. However, any enumeration algorithm
similar to the one presented here that builds lattices by adding
elements one-by-one cannot build only modular lattices if it is
supposed to build all modular lattices (this can be seen for example
by removing any element from the subspace lattice of a finite
projective plane, see Figure~\ref{fano}).
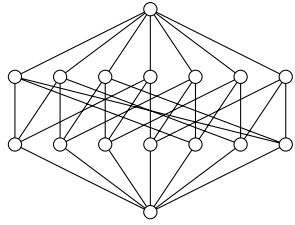
\begin{figure}
\tikzstyle{every node}=[draw,fill=white,circle,inner sep=0pt,minimum size=5pt]
\begin{tikzpicture}[scale=.6]
\draw(-3,1.5)--(-2,3)--(2,1.5)--(3,3)--(0,1.5)--(1,3)--(-2,1.5)--(-1,3)--(3,1.5)--(-3,3)--(1,1.5)--(2,3)--(-1,1.5)--(0,3)--(-3,1.5);
\draw(0,0)node{}--(-3,1.5)node{}--(-3,3)node{}--(0,4.5)node{}--
(-2,3)node{}--(-2,1.5)node{}--(0,0)--(-1,1.5)node{}--(-1,3)node{}--(0,4.5)--
(0,3)node{}--(0,1.5)node{}--(0,0)--(1,1.5)node{}--(1,3)node{}--(0,4.5)--
(2,3)node{}--(2,1.5)node{}--(0,0)--(3,1.5)node{}--(3,3)node{}--(0,4.5);
\end{tikzpicture}
\caption{The subspace lattice of the Fano plane.}\label{fano}
\end{figure}

We would like to thank Michael Fahy and Nikos Hatzopoulus for help
with operating the computing clusters. We also thank the 
Summer Undergraduate Research Fellowship at Chapman University for 
financial support.

\end{document}